\documentclass[12pt]{amsart}
\usepackage[margin=1in]{geometry}

\usepackage{graphicx}
\usepackage{amsthm}
\usepackage{amssymb}
\usepackage{tikz}

\newtheorem{theorem}{Theorem}
\newtheorem{prop}[theorem]{Proposition}
\newtheorem{lemma}[theorem]{Lemma}
\newtheorem{cor}[theorem]{Corollary}

\theoremstyle{definition}
\newtheorem{definition}[theorem]{Definition}
\newtheorem{example}[theorem]{Example}
\theoremstyle{remark}

\numberwithin{theorem}{section}
\numberwithin{equation}{section}

\newcommand{\R}{\mathbb{R}}

\renewcommand{\phi}{\varphi}

\newcommand{\paren}[1]{\left(#1\right)}
\newcommand{\set}[1]{\left\{#1\right\}}
\newcommand{\bracket}[1]{\left[#1\right]}
\newcommand{\abs}[1]{\left\lvert#1\right\rvert}

\newcommand{\Vtil}{\widetilde{V}}

\begin{document}

\title{Double Bubbles on the Line with Log-convex Density $f$ with $(\log f)'$ Bounded}

\author{Nat Sothanaphan}
\address{Courant Institute of Mathematical Sciences\\ New York University\\
New York, NY \\ 10003, USA}
\email{natsothanaphan@gmail.com}

\begin{abstract}
We extend results of Bongiovanni et al. \cite{db} on double bubbles on the line with log-convex density
to the case where the derivative of the log of the density is bounded.
We show that the tie function between the double interval and the triple interval still exists but may blow up to infinity in finite time.
For the first time, a density is presented for which the blowup time is positive and finite.
\end{abstract}

\maketitle

\section{Introduction}
Consider $\R$ with a symmetric, strictly log-convex, $C^1$ density $f$.
Bongiovanni et al. \cite{db} show that a perimeter-minimizing double bubble enclosing volumes $V_1 \leq V_2$ is one of the following two configurations of Figure \ref{fig:doubletriple}:
\begin{itemize}
\item A \emph{double interval}: two contiguous intervals in equilibrium enclosing volumes $V_1$ and $V_2$;
\item A \emph{triple interval}: an interval symmetric about the origin enclosing volume $V_1$ flanked by two intervals on each side, each enclosing volume $V_2/2$.
\end{itemize}
They also show that, if $(\log f)'$ is unbounded, there is a tie function $\lambda(V_1)$ such that
for $V_2=\lambda(V_1)$, the double interval and the triple interval tie (have equal perimeter);
for $V_2>\lambda(V_1)$,  the triple interval is uniquely perimeter minimizing; and
for $V_2<\lambda(V_1)$, the double interval is uniquely perimeter minimizing up to reflection \cite[Thm. 4.15]{db}.

\begin{figure}[ht]
\vspace{-20pt}
\includegraphics[width=200pt]{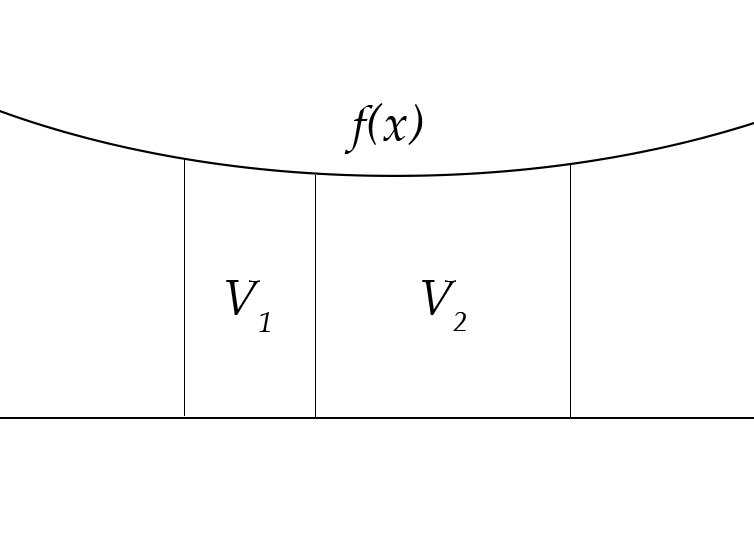}
\hspace{30pt}
\includegraphics[width=200pt]{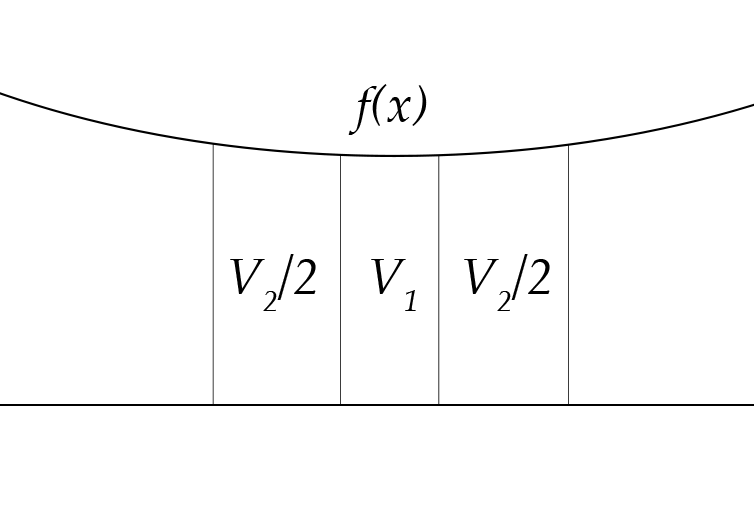}
\vspace{-25pt}
\caption{A double interval and a triple interval on the real line.}
\label{fig:doubletriple}
\end{figure}

The goal of this note is to extend this result to the case where $(\log f)'$ is bounded.
We show that the tie function $\lambda$ still exists, but it may ``blow up in finite time'': it is defined only for
$V_1 <V_0$ for some $V_0$ and approaches infinity as $V_1 \to V_0$. See Figure \ref{fig:tiefunc}.
This proves the conjecture stated at the end of \cite[Section 4]{db}.

\begin{figure}[ht]
\begin{tikzpicture}
\draw [<->, thick] (0,4) node [above left] {$V_2$} -- (0,0) -- (4,0) node [below right] {$V_1$};
\draw (0,0) -- (4,4);
\draw node at (1.4,2.8) {Double};
\node [below] at (2,-0.7) {$V_0=0$};

\begin{scope}[shift={(0.5,0)}]
\draw [<->, thick] (5,4) node [above left] {$V_2$} -- (5,0) -- (9,0) node [below right] {$V_1$};
\draw (5,0) -- (9,4);
\draw [help lines, dashed] (6.7,0) node [below] {$V_0$} -- (6.7,4);
\node [below] at (6.7,0) {$V_0$};
\draw [thick, blue] (5,1) to [out=30,in=265] (6.55,4) node [above] {$\lambda$};
\draw node at (5.7,3) {Triple};
\draw node [rotate=45] at (6.7,2.2) {Double};
\node [below] at (7,-0.7) {$0<V_0<\infty$};
\end{scope}

\begin{scope}[shift={(1,0)}]
\draw [<->, thick] (10,4) node [above left] {$V_2$} -- (10,0) -- (14,0) node [below right] {$V_1$};
\draw (10,0) -- (14,4);
\draw [thick, blue] (10,1.4) to [out=30,in=235] (12.8,4) node [above] {$\lambda$};
\draw node at (11,3) {Triple};
\draw node [rotate=45] at (11.8,2.3) {Double};
\node [below] at (12,-0.7) {$V_0=\infty$};
\end{scope}
\end{tikzpicture}
\caption{Three possibilities for the tie function $\lambda$ between the double interval and the triple interval:
$\lambda$ not existing, $\lambda$ blowing up to infinity in finite time, and $\lambda$ existing for all time.
Each case occurs for some symmetric, strictly log-convex, $C^1$ density (Ex. \ref{ex:threetypes}).}
\label{fig:tiefunc}
\end{figure}
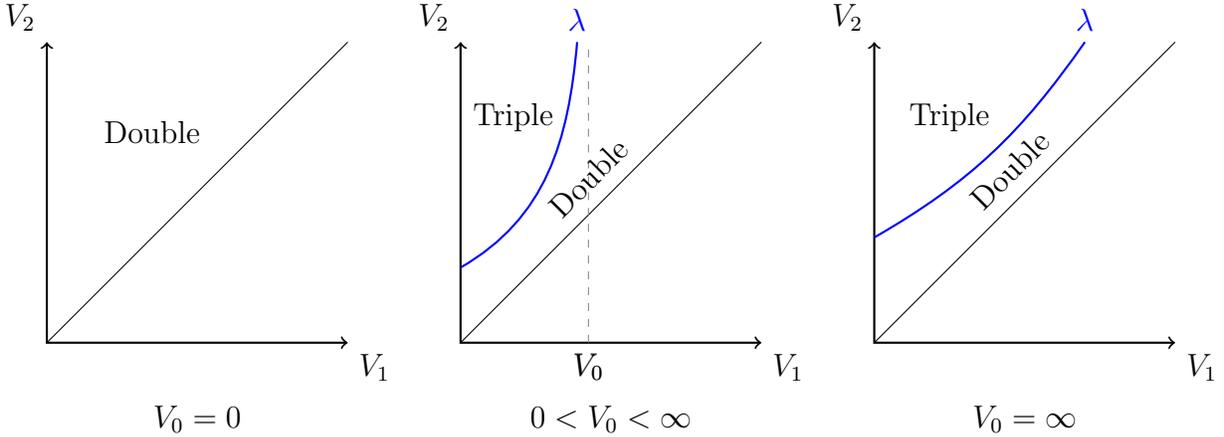

Our main result is as follows.

\begin{theorem}
\label{thm:main}
Consider $\R$ with a symmetric, strictly log-convex, $C^1$ density.
There exists a ``blowup time'' $0\leq V_0 \leq \infty$ such that, for each $V_1<V_0$,
there is a unique $\lambda(V_1)>V_1$ with the following properties.
\begin{itemize}
\item For $V_1<V_0$ and $V_2=\lambda(V_1)$, the double interval and the triple interval tie. \\
\item For $V_1 <V_0$ and $V_2>\lambda(V_1)$, the perimeter-minimizing double bubble is uniquely the triple interval. \\
\item For either $V_1 \geq V_0$ or $V_1<V_0$ and $V_1 \leq V_2<\lambda(V_1)$, the perimeter-minimizing double bubble is uniquely the double interval up to reflection.
\end{itemize}
See Figure \ref{fig:tiefunc}.
Moreover, each of the three types of blowup: $V_0=0$, $0<V_0<\infty$, and $V_0=\infty$, occurs for some symmetric, strictly log-convex, $C^1$ density.
\end{theorem}

Proposition \ref{prop:existblowup} gives more properties of the tie function $\lambda$.
We provide a way to compute the blowup time $V_0$ in Proposition \ref{prop:computeblowup},
and criteria for when the blowup time is infinity or zero in Corollaries \ref{cor:criteriainf} and \ref{cor:criteria0}.
Finally, Example \ref{ex:threetypes} presents densities exhibiting the three types of blowup.
In particular, we show that it is possible for the blowup time to be positive and finite.

This note is organized as follows.
Section \ref{sec:existence} shows the existence of the blowup time $V_0$ and tie function $\lambda$ in Theorem \ref{thm:main} by modifying the proof of \cite[Thm. 4.15]{db}.

In Section \ref{sec:compute}, we derive a formula for the blowup time $V_0$ in Proposition \ref{prop:computeblowup}. This result is then used to develop criteria for when $V_0=\infty$ (Cor. \ref{cor:criteriainf}) and $V_0=0$ (Cor. \ref{cor:criteria0}).
Finally, we present examples of densities for which $V_0=0$, $0<V_0<\infty$, and $V_0=\infty$ in Example \ref{ex:threetypes}.

\section{Existence of Blowup Time}
\label{sec:existence}
We establish the existence of the blowup time $V_0$ and tie function $\lambda$ in Theorem \ref{thm:main}.
Our notation follows Bongiovanni et al. \cite{db}.
For prescribed volumes $V_1 \leq V_2$, let $\mu(V_1,V_2)=P_3-P_2$
be the difference of perimeters of the triple interval and the double interval.
In Bongiovanni et al. \cite{db}, Proposition 4.11 requires the extra hypothesis that $(\log f)'$ is unbounded,
but Lemmas 4.8, 4.9, and 4.14 and Proposition 4.10 do not
and hold for any symmetric, strictly log-convex, $C^1$ density.

The following quantity is useful in characterizing the blowup time.

\begin{definition}
\label{def:muell}
Consider $\R$ with a symmetric, strictly log-convex, $C^1$ density.
For each $V_1$, define $\mu_\ell(V_1)$ to be
$$\mu_\ell(V_1) = \lim_{V_2 \to \infty} \mu(V_1,V_2),$$
the limit of the difference of perimeters of the triple interval and the double interval as $V_2 \to \infty$.
\end{definition}

Notice that $\mu_\ell$ is well defined because $\mu$ is strictly decreasing in $V_2$ \cite[Lemma 4.9]{db},
with $-\infty \leq \mu_\ell < \infty$. Moreover, since $\mu$ is strictly increasing in $V_1$ \cite[Lemma 4.9]{db},
$\mu_\ell$ is nondecreasing.

We can now state the following proposition, which proves the existence of the blowup time $V_0$ and tie function $\lambda$.

\begin{prop}
\label{prop:existblowup}
Consider $\R$ with a symmetric, strictly log-convex, $C^1$ density.
There exists a ``blowup time'' $0\leq V_0 \leq \infty$ such that, for each $V_1<V_0$,
there is a unique $\lambda(V_1)>V_1$ with the following properties.
\begin{itemize}
\item For $V_1<V_0$ and $V_2=\lambda(V_1)$, the double interval and the triple interval tie,
and they are the only perimeter-minimizing double bubbles up to reflection. \\
\item For $V_1 <V_0$ and $V_2>\lambda(V_1)$, the perimeter-minimizing double bubble is uniquely the triple interval. \\
\item For either $V_1 \geq V_0$ or $V_1<V_0$ and $V_1 \leq V_2<\lambda(V_1)$, the perimeter-minimizing double bubble is uniquely the double interval up to reflection. \\
\item $\lambda$ is strictly increasing, $C^1$, tends to infinity as $V_1 \to V_0$, and tends to a positive limit as $V_1 \to 0$.
\end{itemize}
\end{prop}

\noindent See Figure \ref{fig:tiefunc}.

\begin{proof}
We modify the proof of \cite[Thm. 4.15]{db}.
The idea is that $\mu_\ell$ (Def. \ref{def:muell}) should be negative in the region where $\lambda$ is defined.
With this in mind, let the blowup time $V_0$ be
\begin{equation}
\label{eq:constructV0}
V_0 = \sup \{V_1: \mu_\ell(V_1) < 0 \}, \quad 0 \leq V_0 \leq \infty,
\end{equation}
where this quantity is zero if no $V_1$ satisfies the condition.

For $V_1 <V_0$, we now construct the tie function $\lambda(V_1)$.
Because $\mu_\ell$ is nondecreasing, $\mu_\ell(V_1)<0$.
Since $\mu(V_1,V_1)>0$ \cite[Prop. 4.10]{db} and $\mu$ is strictly decreasing in $V_2$
\cite[Lemma 4.9]{db}, there is a unique $\lambda(V_1) > V_1$ such that $\mu(V_1,\lambda(V_1))=0$.
Then $\mu(V_1,V_2)<0$ for $V_2>\lambda(V_1)$ and $\mu(V_1,V_2)>0$ for $V_1 \leq V_2 < \lambda(V_1)$.
Because the double interval and the triple interval are the only possible perimeter minimizers \cite[Prop. 4.6]{db},
we have proved the first three items in the case that $V_1<V_0$.

Consider now the case $V_1\geq V_0$. We must show that $\mu(V_1,V_2)>0$ for all $V_2 \geq V_1$.
If $V_0=\infty$, this is trivial.
If $V_0=0$, then $\mu_\ell(V_1)\geq 0$ for all $V_1>0$.
Because $\mu$ is strictly decreasing in $V_2$,
$\mu(V_1,V_2) > 0$ for all $V_1 \leq V_2$, as desired.
Now suppose that $0<V_0<\infty$.
We claim that $\mu_\ell(V_0) \geq 0$.
If $\mu_\ell(V_0)<0$, then for some $V_2>V_0$, $\mu(V_0,V_2)<0$.
By continuity of $\mu$, for some $V_0<V_1<V_2$, $\mu(V_1,V_2)<0$, and so $\mu_\ell(V_1)<0$, contradiction.
Thus the claim holds. Becaue $\mu_\ell$ is nondecreasing, $\mu_\ell(V_1) \geq 0$ for all $V_1 \geq V_0$.
Then because $\mu$ is strictly decreasing in $V_2$,
$\mu(V_1,V_2) > 0$ for all $V_2 \geq V_1 \geq V_0$.

It remains to prove the fourth item. Using exactly the same arguments as in the proof of \cite[Thm. 4.15]{db}, one can show everything except the statement that $\lambda(V_1) \to \infty$ as $V_1\to V_0$.
We now show this last statement. Notice that we can assume $0<V_0<\infty$.
Suppose to the contrary that $\lambda$ increases to a finite limit $L$ as $V_1 \to V_0$.
Then $\mu(V_1,V_2)<0$ for all $V_1<V_0$ and $V_2>L$. By continuity, $\mu(V_0,V_2) \leq 0$ for all $V_2>L$,
and so because $\mu$ is strictly decreasing in $V_2$, $\mu(V_0,V_2)<0$ for all $V_2>L$.
Again by continuity, $\mu(V_1,V_2)<0$ for some $V_1>V_0$ and $V_2>L$,
implying that $\mu_\ell(V_1)<0$, contradiction.
Therefore, the proposition is proved.
\end{proof}

A question remains: what values can the blowup time $V_0$ take?
Theorem 4.15 of Bongiovanni et al. shows that if $(\log f)'$ is unbounded, then $V_0=\infty$,
and their Example 4.13 gives a density with $V_0=0$.
In the next section, we will devise a general procedure for determining the value of $V_0$
and finally present a density for which $0<V_0<\infty$.

\section{Computing Blowup Time}
\label{sec:compute}

We now seek to compute the blowup time $V_0$ in Proposition \ref{prop:existblowup}.
We must first define some quantities.
As in Bongiovanni et al. \cite[Lemma 4.2]{db}, define the volume coordinate by
$$V = \int_0^x f,$$
where $x$ is the positional coordinate.
In particular, $f$ is a strictly log-convex density if and only if $f$ is strictly \emph{convex} in volume coordinate.

\begin{definition}
\label{def:LM}
On $\R$ with a symmetric, strictly log-convex, $C^1$ density $f$, define
\begin{align*}
L& =\lim_{x\to\infty} (\log f)'(x) = \lim_{V \to \infty} f'(V), \\
M& =\lim_{V \to \infty} f(2V)-2f(V),
\end{align*}
where $x$ is the positional coordinate and $V$ the volume coordinate.
\end{definition}

Notice that $L$ exists because $f'$ is strictly increasing and $M$ exists because, by taking derivatives, $f(2V)-2f(V)$ is strictly increasing when $V>0$.
Observe that $0<L\leq \infty$ and $-\infty < M \leq \infty$.

From now on, we work exclusively with volume coordinates.
The following lemma shows a relationship between $L$ and $M$.

\begin{lemma}
\label{lem:LM}
$L=\infty$ implies $M=\infty$.
\end{lemma}

\begin{proof}
Let $g(V)=f(2V)-2f(V)$ for $V >0$. Because $g$ is strictly increasing, $g \leq M$.
Notice that
$$\frac{f(2V)}{2V}-\frac{f(V)}{V} = \frac{g(V)}{2V},$$
so by telescoping,
$$\frac{f(2^n)}{2^n}-f(1) = \sum_{k=0}^{n-1}\frac{g(2^k)}{2^{k+1}}
\leq \paren{1-\frac{1}{2^n}}M.$$
Since $L=\infty$, $f(V)/V \to \infty$ as $V \to \infty$.
So as $n\to\infty$, the left-hand side diverges to infinity, implying that $M=\infty$.
\end{proof}

The following example shows that the converse of Lemma \ref{lem:LM} is not true.

\begin{example}
Consider $f(V)=V\tan^{-1} V - \log(V^2+1)/2+1$.
We can easily check that $f$ is symmetric, $C^1$ with $f'(V)=\tan^{-1} V$, and strictly convex.
Therefore $L=\pi/2$, and we can compute that $M=\infty$.
\end{example}

We now define another quantity, which will later be shown to be related to the left endpoint of the double interval.

\begin{definition}
\label{def:Vstar}
Let $f$ be a symmetric, strictly log-convex, $C^1$ density on $\R$ with $f'(V)$ bounded.
For each $V_1$, define $V^*$ to be the unique solution to
\begin{equation}
\label{eq:Vstar}
f'(V^*)+f'(V^*+V_1)=-L,
\end{equation}
where $L$ is as in Definition \ref{def:LM}.
\end{definition}

Notice that the left-hand side of \eqref{eq:Vstar} is strictly increasing in $V^*$, tends to $-2L$ as $V^* \to -\infty$,
and tends to $2L$ as $V^* \to \infty$. So \eqref{eq:Vstar} has a unique solution, and $V^*$ is well defined.

Let $\Vtil$ be the leftmost endpoint of the double interval. By the equilibrium condition \cite[Cor. 3.3]{db}, $\Vtil$ is the unique solution to
\begin{equation}
\label{eq:equilibrium}
f'(\Vtil)+f'(\Vtil+V_1)+f'(\Vtil+V_1+V_2)=0.
\end{equation}
Because $f'$ is strictly increasing, we can see that $\Vtil$ is strictly decreasing in both $V_1$ and $V_2$.
We now characterize $V^*$ as the limit of $\Vtil$ as $V_2 \to \infty$.

\begin{lemma}
\label{lem:limitVstar}
Suppose that $L<\infty$.
For a fixed $V_1$, $\lim_{V_2 \to \infty} \Vtil = V^*$.
\end{lemma}

\begin{proof}
Since $\Vtil$ is strictly decreasing in $V_2$, the limit $V_\ell=\lim_{V_2 \to \infty} \Vtil$ exists (it may be $-\infty$).
It remains to show that $V_\ell=V^*$.
By \cite[Lemma 4.8]{db}, $\Vtil > -(V_1+V_2)/2$, so $\Vtil+V_1+V_2 > (V_1+V_2)/2 \to \infty$
as $V_2 \to \infty$.
Hence by taking $V_2 \to \infty$ in \eqref{eq:equilibrium},
$$f'(V_\ell)+f'(V_\ell+V_1) +L = 0,$$
where we interpret $f'(-\infty) = -L$ in the case that $V_\ell=-\infty$.
Finally, observe that $V_\ell=-\infty$ is not possible due to $L>0$, so $V_\ell$ is finite and equals $V^*$.
\end{proof}

The next lemma collects some properties of $V^*$ as $V_1 \to \infty$.

\begin{lemma}
\label{lem:VstarV1big}
Suppose that $L<\infty$ and let $V^*$ be as in Definition \ref{def:Vstar}. Then $V^*$ is strictly decreasing in $V_1$, $\lim_{V_1 \to \infty} V^* = -\infty$, and $\lim_{V_1 \to \infty} (V^*+V_1) = 0$.
\end{lemma}

\begin{proof}
The fact that $f'$ is strictly increasing implies that $V^*$ is strictly decreasing in $V_1$.
This and the fact that $V^* \leq -V_1$ \cite[Lemma 4.8]{db} imply that $\lim_{V_1 \to \infty} V^* = -\infty$.
Now take $V_1 \to\infty$ in \eqref{eq:Vstar} to obtain $\lim_{V_1 \to \infty} (V^*+V_1) = 0$.
\end{proof}

The following proposition gives a way to compute $\mu_\ell$ (Def. \ref{def:muell}) based on Lemma \ref{lem:limitVstar}.

\begin{prop}
\label{prop:computemuell}
Let $\mu_\ell$ be as in Definition \ref{def:muell}, $L$ and $M$ be as in Definition \ref{def:LM}, and $V^*$ be as in Definition \ref{def:Vstar}.
Suppose that $L<\infty$. Then
\begin{equation}
\label{eq:muell}
\mu_\ell(V_1) = 2f\paren{\frac{V_1}{2}}-f(V^*)-f(V^*+V_1)-V^*L-M,
\end{equation}
where this quantity is finite if $M<\infty$ and equals $-\infty$ if $M=\infty$.
\end{prop}

\begin{proof}
By Lemma \ref{lem:limitVstar},
\begin{align*}
\mu_\ell(V_1)&=\lim_{V_2 \to\infty} (P_3-P_2) \\
&=\lim_{V_2 \to\infty} \bracket{2f\paren{\frac{V_1}{2}} + 2f\paren{\frac{V_1+V_2}{2}} - f(\Vtil) - f(\Vtil+V_1) - f(\Vtil+V_1+V_2)} \\
&= 2f\paren{\frac{V_1}{2}}-f(V^*)-f(V^*+V_1)+
\lim_{V_2\to\infty}\bracket{2f\paren{\frac{V_1+V_2}{2}} - f(\Vtil+V_1+V_2)}.
\end{align*}
Rewrite the quantity in the last limit as
$$\bracket{2f\paren{\frac{V_1+V_2}{2}} - f(V_1+V_2)} +\bracket{ f(V_1+V_2) - f(\Vtil+V_1+V_2)}.$$
The first bracket tends to $-M$ as $V_2 \to \infty$.
By the Mean Value Theorem, the second bracket equals
$-\Vtil f'(V)$
for some $V>\Vtil+V_1+V_2>(V_1+V_2)/2 \to \infty$ as $V_2 \to\infty$ \cite[Lemma 4.8]{db}, so it tends to $-V^*L$ as $V_2 \to \infty$.
Therefore, $\mu_\ell$ has the desired formula.
\end{proof}

We are now ready to state the proposition computing the blowup time $V_0$.

\begin{prop}
\label{prop:computeblowup}
Consider $\R$ with a symmetric, strictly log-convex, $C^1$ density $f$.
Let $L$ and $M$ be as in Definition \ref{def:LM} and $V^*$ be as in Definition \ref{def:Vstar}.
Then the blowup time $V_0$ of Proposition \ref{prop:existblowup} can be computed as follows.

\begin{itemize}
\item If $L=\infty$ or $M=\infty$, then $V_0=\infty$.
\item If $L<\infty$ and $M<\infty$, then $V_0<\infty$ and
\begin{equation}
\label{eq:condV0}
V_0=\sup\set{V_1: \mu_\ell(V_1) <0} = \inf\set{V_1: \mu_\ell(V_1) \geq 0},
\end{equation}
where $\mu_\ell$ has the formula \eqref{eq:muell}
and the $\sup$ is $0$ if there is no $V_1$ satisfying its condition.
\end{itemize}
\end{prop}

\begin{proof}
Bongiovanni et al. \cite[Prop. 4.11]{db} show that if $L=\infty$, then $V_0=\infty$.
So suppose that $L<\infty$. By Proposition \ref{prop:computemuell}, $\mu_\ell$ is given by \eqref{eq:muell}.

The characterization of $V_0$ in \eqref{eq:constructV0} shows the first half of \eqref{eq:condV0}.
Then fact that $\mu_\ell$ is nondecreasing implies the second half of \eqref{eq:condV0}.

It remains to show that, still assuming $L<\infty$, $V_0=\infty$ if and only if $M=\infty$.
If $M=\infty$, then $\mu_\ell(V_1)=-\infty$ by \eqref{eq:muell}, and so $V_0=\infty$.
Now suppose that $M<\infty$.
By \eqref{eq:muell}, we can write
$$\mu_\ell(V_1) = \bracket{2f\paren{\frac{V_1}{2}}-f(V_1)}+\bracket{f(-V_1)-f(V^*)}-f(V^*+V_1)-V^*L-M.$$
Take $V_1 \to \infty$ and apply Lemma \ref{lem:VstarV1big}.
The first bracket tends to $-M$. The second bracket is $(-V_1-V^*)f'(V)$ for some $V$, which tends to 0 because $V^*+V_1\to 0$ as $V_1 \to \infty$ and $f'$ is bounded. Finally, $f(V^*+V_1) \to f(0)$ and $V^*L \to -\infty$.
Hence $\mu_\ell \to \infty$ as $V_1 \to \infty$, showing that $V_0<\infty$.
\end{proof}

From Proposition \ref{prop:computeblowup}, we obtain two corollaries stating conditions for when $V_0=0$ and $V_0=\infty$.

\begin{cor}
\label{cor:criteriainf}
Consider $\R$ with a symmetric, strictly log-convex, $C^1$ density. Let $L$ and $M$ be as in
Definition \ref{def:LM}. Then $V_0=\infty$ if and only if $M=\infty$. In particular, $L=\infty$ implies $V_0=\infty$.
\end{cor}

\begin{proof}
By Proposition \ref{prop:computeblowup}, $V_0=\infty$ if and only if $L=\infty$ or $M=\infty$.
This is equivalent to $M=\infty$ by Lemma \ref{lem:LM}.
\end{proof}

\begin{cor}
\label{cor:criteria0}
Consider $\R$ with a symmetric, strictly log-convex, $C^1$ density $f$. Let $L$ and $M$ be as in
Definition \ref{def:LM}. Then $V_0=0$, that is, the double interval is uniquely perimeter minimizing for all prescribed volumes if and only if $L<\infty$ and
\begin{equation}
\label{eq:V00}
2f(0)-2f(V)+VL -M\geq 0, \quad \text{where } V = (f')^{-1}(L/2).
\end{equation}
\end{cor}

\begin{proof}
By Corollary \ref{cor:criteriainf}, $L=\infty$ implies $V_0=\infty$. So suppose that $L<\infty$.
By Proposition \ref{prop:computeblowup}, $V_0=0$ if and only if $\mu_\ell(V_1) \geq 0$ for all $V_1>0$.
Since $\mu_\ell$ is nondecreasing, this is equivalent to $\lim_{V_1 \to 0} \mu_\ell(V_1) \geq 0$.
From \eqref{eq:Vstar}, $V^* \to -(f')^{-1}(L/2)$ as $V_1 \to 0$.
Hence by \eqref{eq:muell}, $\lim_{V_1 \to 0} \mu_\ell(V_1)$ equals the left-hand side of \eqref{eq:V00},
proving the corollary.
\end{proof}

Finally, the following example shows densities with the three types of blowup: $V_0=0$, $0<V_0<\infty$, and $V_0=\infty$. In particular, it shows that the case $0<V_0<\infty$ is indeed possible.

\begin{example}
\label{ex:threetypes}
All densities below are symmetric, $C^1$, and strictly convex in volume coordinate.
\begin{itemize}
\item \cite[Ex. 4.13]{db}. Consider $f(V)=\abs{V}+e^{-\abs{V}}$.
We can compute $L=1$, $M=0$, and $(f')^{-1}(L/2)=\log 2$.
By Corollary \ref{cor:criteria0}, we can check that $V_0=0$.

\item Consider $f(V)=\sqrt{V^2+1}-1/2$.
Then $L=1$, $M=1/2$, and $(f')^{-1}(L/2)=1/\sqrt{3}$.
Corollaries \ref{cor:criteriainf} and \ref{cor:criteria0} imply that $0<V_0<\infty$.

\item Consider the Borell density $f(x)=e^{x^2}$. Then $(\log f)'(x) = 2x$ is unbounded, so $L=\infty$.
By Corollary \ref{cor:criteriainf}, $V_0=\infty$.
\end{itemize}
\end{example}

With this example, we have completed the proof of Theorem \ref{thm:main}.

\section*{Acknowledgements}
I would like to thank Frank Morgan for checking the details and giving suggestions that improve the exposition of this paper.

\medskip

\noindent MSC2010: 49Q10

\medskip

\noindent Key words and phrases: double bubble, density, isoperimetric

\end{document}